\documentclass[smallextended]{svjour3}       % onecolumn (second format)
\smartqed  % flush right qed marks, e.g. at end of proof

\usepackage{graphicx,epsfig}
\usepackage{amsmath,mathrsfs} 
\usepackage{amssymb} 
\usepackage{amsfonts}
\usepackage{mathtools}
\usepackage{epstopdf}
\usepackage{ifthen}
\usepackage{bbm}
\graphicspath{{eps/}}

\usepackage{graphicx,color,tikz}

\newcommand{\inR}{\in \mathbb{R}}

\newcommand{\R}{ \mathbb{R}}
\newcommand{\Z}{ \mathbb{Z}}

\newcommand{\eqdef}{\stackrel{\vartriangle}{=}}

\newcommand{\Lop}{{\rm L}}

\newcommand{\dint}{{\rm d}}

\newcommand{\Fourier}{ \mathcal{F}}

\newcommand{\bw}{{\boldsymbol \omega}}

\def\V#1{{\boldsymbol{#1}}}         % vectors
\def\Spc#1{{\mathcal{#1}}}  % spaces
\def\M#1{{\bf{#1}}}  % matrices
\def\Op#1{{\mathrm{#1}}}  % operator

\newcommand{\toC}{\xrightarrow{\mbox{\tiny \ \rm c.  }}}
\renewcommand{\[}{\begin{equation}}
\renewcommand{\]}[1]{\label{eq:#1}\end{equation}}

\newcommand{\ie}{{\em i.e., }}

\providecommand{\revise}[1]{{#1}}
\title{A unifying representer theorem for inverse problems and machine learning \thanks{The research leading to these results has received funding from the Swiss National Science Foundation under Grant 200020-162343.}}

\author{
Michael Unser  
 }
 
 \institute{Biomedical Imaging Group, \'Ecole polytechnique f\'ed\'erale de Lausanne (EPFL),
Station 17, CH-1015, Lausanne, Switzerland\\              \email{michael.unser@epfl.ch}           %  \\
}

\begin{document}

\date{Received: December 4, 2019 / Accepted: July 9, 2020}

% the following settings can be set or left blank at first

%\frontmatter  % title page, contents, catalog information

\maketitle
%\showindexmarks
\begin{abstract}
Regularization addresses the ill-posedness of the training problem in machine learning or the reconstruction of a signal from a limited number of measurements. The method is applicable whenever the problem is formulated as an optimization task.
The standard strategy consists in augmenting the original cost functional by an energy that penalizes solutions with undesirable behavior.
The effect of regularization is very well understood when the penalty involves a Hilbertian norm. Another popular configuration is the use of an $\ell_1$-norm (or some variant thereof)
that favors sparse solutions. In this paper, we propose a higher-level formulation of regularization
within the %general 
context of Banach spaces. We present a general representer theorem that characterizes the solutions of a remarkably broad class of optimization problems. We then use our theorem to retrieve a number of known results in the literature such as the celebrated representer theorem of machine leaning for RKHS, Tikhonov regularization, representer theorems for sparsity promoting functionals, the recovery of spikes, as well as a few new ones.
\end{abstract}

\subclass{46N10 \and 47A52 \and 65J20 \and 68T05}

\keywords{convex optimization \and regularization \and representer theorem \and inverse problem \and machine learning \and Banach space}

\vspace{1em}
\noindent Communicated by Tomaso Poggio

%\tableofcontents
%
%\mainmatter

\section{Introduction}
A recurrent problem in science and engineering is the reconstruction of a multidimensional signal $f: \R^d \to \R$ from a finite number of (possibly noisy) linear measurements $\V y=(y_m)=\V \nu(f) \in \R^M$, where the operator 
$\V \nu=(\nu_m): f \mapsto \V \nu(f)=(\langle \nu_1,f \rangle,\dots,\langle \nu_M, f\rangle )$ symbolizes
the linear measurement process.
The machine-learning version of the problem is the determination of
a function $f: \R^d \to \R$ from a finite number of samples $y_m=f(\V x_m)+\epsilon_m$
where $\epsilon_m$ is a small perturbation term; it is a special case of the former
with $\nu_m=\delta(\cdot-\V x_m)$.
Since a function that takes values over the continuum is an infinite-dimensional entity, the reconstruction problem is inherently ill-posed. 

The standard remedy is to impose an additional minimum-energy requirement which, in effect, regularizes the solution. A natural choice of regularization is a smoothness norm associated with some function space $\Spc X'$ (typically, a Sobolev space), which results in the prototypical formulation of the problem as
\begin{align}
\label{Eq:standardInterp}
S=\arg \min_{f \in \Spc X'} \|f\|_{\Spc X'} \quad \mbox{s.t.} \quad  \langle \nu_m,f\rangle=y_m, \ m=1,\dots,M.
\end{align}
An alternative version that is better suited for noisy data is
\begin{align}
\label{Eq:standardProb}
S=\arg \min_{f \in \Spc X'} \sum_{m=1}^M \left|y_m -\langle \nu_m,f\rangle\right|^2 + \lambda \|f\|^p_{\Spc X'}
\end{align}
with an adequate choice of hyper-parameters $\lambda\in \R^+$ and $p\in[1,\infty)$. We note that the unconstrained form \eqref{Eq:standardProb} %of the problem 
is
a generalization of \eqref{Eq:standardInterp}: the latter is recovered in the limit by taking $\lambda\to0$.

The term ``representer theorem'' is typically used to designate a parametric formu\-la---preferably, a linear expansion in terms of some basis functions---that spans the whole range of solutions, 
%of this type of problem, 
irrespective of the value of the data $\V y \inR^M$.
%Such parametric representations are valued in practice because they often provide the basis for an efficient numerical optimization. 
%The fact that representer theorems are available for certain types of infinite-dimensional optimization problem is in itself already quite remarkable.
Representer theorems are valued by practitioners because they indicate the way in which the initial problem can be recast as a finite-dimensional optimization, making it amenable to %standard 
numerical computations. The other benefit is that the %explicit 
description %explicit form 
of the manifold of possible solutions %, they 
provides one with a better understanding of the %intrinsic 
effect of regularization.
%
%First, it is a strong indication that the problem is amenable to standard numerical computations. Indeed, it often indicates the way in which the initial problem can be recast as a finite-dimensional optimization.
%Second, by describing the manifold of possible solutions, it provides us with an understanding of the intrinsic effect of regularization. 
The best known example is 
the representer theorem for reproducing-kernel Hilbert spaces (RKHS), %is a general result by Sch\"olkopf and others 
which states that the solution of \eqref{Eq:standardProb} with
$\langle \nu_m,f\rangle=f(\V x_m)$ and a Hilbertian regularization norm
%induced by an inner product %(i.e., $\|f\|^2_{\Spc X'}=\langle f,f\rangle_{\Spc H}$ where $\Spc X'=\Spc H$ is a reproducing kernel Hilbert space)
necessarily lives in a subspace of dimension $M$ spanned by kernels centered on the data coordinates $\V x_m$ \cite{deBoor1966} \cite{Kimeldorf1971c} \cite{Poggio1990} \cite{Poggio1990b} \cite{Scholkopf2002}. 
%\cite{deBoor1966,Kimeldorf1971c,Poggio1990,Scholkopf2002}. 
This theorem, in its extended version \cite{Scholkopf2001}, is the foundation for the majority of kernel-based methods for machine learning, including regression, radial-basis functions, and support-vector machines \cite{Scholkopf1997} \cite{Evgeniou2000} \cite{Steinke2008}. There is also a whole line of generalizations of the concept that involves reproducing kernel Banach spaces (RKBS) \cite{Zhang2009} \cite{Zhang2012b} \cite{Xu2019}.
More recently, motivated by the success of $\ell_1$ and total-variation regularization for compressed sensing \cite{Donoho2006} \cite{Candes2007} \cite{Bruckstein2009}, researchers have derived alternative 
representer theorems in order to explain the sparsifying effect of such penalties and their robustness to missing data \cite{Foucart2013} \cite{Unser2016} \cite{Gupta2018} \cite{Boyer2018}. A representer theorem for measures has also been invoked to justify the use of the total-variation %of measures 
norm for
the super-resolution localization of spikes \cite{Candes2013b} \cite{Denoyelle2017} \cite{Duval2015} \cite{Poon2019} (see, Section \ref{Sec:Spikes} for  details).

In this paper, we present a unifying treatment of regularization by considering the problem
from the abstract perspective of optimization in Banach spaces.
Our motivation there is essentially two-fold: (1) to get a better
``geometrical'' understanding of the effect of regularization, and (2) 
to state a generic representer theorem that applies to a wide variety of objects describable as elements of some native Banach space.
The supporting theory is developed in Section \ref{Sec:Theory}.
%The first part of the paper (Section \ref{Sec:Theory}) is devoted to the theory. 
Our formulation takes advantage of the notion of
Banach conjugates which is explained in Section \ref{Sec:Duality}.
We then immediately proceed with the presentation of our key result:
a  generalized representer theorem (Theorem \ref{Theo:GeneralRepBanach})  
that is valid for arbitrary convex data terms and Banach spaces in general, including the non-reflexive ones.
The proof that is developed Section \ref{Sec:GenTheo} is rather soft (or ``high-level"), as it relies exclusively on the powerful machinery of duality mappings and the Hahn-Banach theorem---in other words, there is no need for G\^ateaux derivatives nor subdifferentials, which are often invoked in such contexts.
%In essence, the first element
%specifies a (dual) solution manifold that is intrinsically linear, while the second projects it back
%into primary space via the duality mapping, which may or may not be linear, depending on wether the regularization norm is Hilbertian or not.
%
%
%While the argument that ensures the existence is fairly standard,
%the part that is less so %---at least not to the great majority of practitioners---
%is the explicit characterization of the solution via the powerful tool of duality mapping.
%The main advantage of our high-level formulation is that it clarifies the issue by concentrating on the essentials---namely, the possibility of saturating the underlying norm inequalities---without any recourse to infinite-dimensional differential calculus (e.g., the evaluation of G\^ateaux derivatives or subdifferentials).
The resulting form of the solution in Theorem \ref{Theo:GeneralRepBanach} is enlightening because it separates out the effect of
the measurement operator from that of the regularization topology. Specifically,
the measurement functionals $\nu_1,\dots,\nu_M$ in \eqref{Eq:standardInterp} or \eqref{Eq:standardProb}
specify a linear solution manifold that is then isometrically mapped into the primary space via the conjugate map
$\Op J_\Spc X: \Spc X \to \Spc X'$, which may or may not be linear, depending on wether the regularization norm is Hilbertian or not.

The theory is then complemented with concrete examples of usage of Theorem \ref{Theo:GeneralRepBanach} to illustrate the power of the approach as well as its broad range of applicability.
Section \ref{Sec:StrictConvex} is devoted to the scenario where the regularization norm is strictly convex, which ensures that the solution of the underlying minimization problem is unique.  We make the link with the existing literature by deriving of a number of classical results: Sch\"olkopf's generalized representer theorem for RKHS (Section \ref{Sec:RKHS}), the closed-form solution of continuous-domain Tikhonov regularization with a Hilbertian norm (Section \ref{Sec.Tik}), and the connection with the theory of reproducing kernel Banach spaces (Section \ref{Sec:RKBS}).
In addition, we present a novel representer theorem for $\ell_p$-norm regularization (Section \ref{Sec:ellp}).
Then, in Section  \ref{Sec:Sparse}, we turn our attention to sparsity promoting regularization which is more challenging because the underlying Banach spaces are typically non-reflexive and non-convex. The enabling ingredient there is a recent result by Boyer {\em et al.\ }\cite{Boyer2018}, which allows one to express the extreme points of the solution set in Theorem \ref{Theo:GeneralRepBanach} as a linear combination of a few basic atoms that are selected adaptively (Theorem \ref{Prop:sumofextremes}).
This result, in its simplest incarnation with $\Spc X'=\ell_1(\Z)$, supports the well-documented sparsifying effect of $\ell_1$-norm minimization, which is central to the theory of compressed sensing. By switching to a continuum, we obtain the representer theorem for $\Spc X'=\Spc M(\Omega)$---the space of signed Radon measures on a compact domain $\Omega$---(Section \ref{Sec:Spikes}), which is relevant to super-resolution localization. We then also derive a representer theorem for generalized total-variation (Section \ref{Sec:SparseKernel})---in the spirit of \cite{Unser2017}---that justifies the use of sparse kernel expansions for machine learning, in line with the \revise{generalized LASSO \cite{Roth2004}.}

\section{Mathematical Formulation}
\label{Sec:Theory}
\subsection{Banach Spaces and Duality Mappings}
The notion of Banach space---basically, a vector space equipped with a norm---is remarkably general. Indeed, the elements (or points) of a Banach space can be vectors (e.g., $\V v\in\R^N$), functions (e.g., $f\in L_2(\R^d)$), sequences (e.g., 
$u[\cdot]\in \ell_1(\Z)$), continuous linear functionals (e.g., $f \in \Spc X'$ where $\Spc X'$ is the dual of some primary Banach space), vector-valued functions (e.g., $\V f=(f_1,\dots,f_N)$ with 
$f_n \in L_2(\R^d)$), matrices $(e.g., \M X \inR^{N \times N}$), and, even,
bounded linear operators from a Banach space $\Spc U$ (domain) to another Banach space $\Spc V$ (range) (e.g., $\Op X \in \Spc L(\Spc U, \Spc V)$) \cite{Megginson1998}.

\label{Sec:Duality}
\begin{definition} A normed vector space $\Spc X$ is 
a linear space equipped with a norm, henceforth denoted by $\|\cdot\|_{\Spc X}$.
It is called a Banach space if it is complete in the sense that every Cauchy sequence
in $(\Spc X,\|\cdot\|_{\Spc X})$ converges to an element of $\Spc X$.
It is said to be 
{\em strictly convex} if,
for all $v_1,v_2 \in \Spc X$ such that $\|v_1\|_{\Spc X}=\|v_2\|_{\Spc X}=1$ and $v_1\ne v_2$, one has that ${\|\lambda v_1+(1-\lambda)v_2\|_{\Spc X}}<1$
for any $\lambda\in(0,1)$. Finally, a Hilbert space is a Banach space whose norm is induced by an inner product.
%A Banach space $(\Spc X,\|\cdot\|_{\Spc X})$
%is a complete normed vector space. 

\end{definition}

We recall that $\Spc X'$ (the continuous dual of $\Spc X$) is the space of linear functionals $u : v \mapsto \langle u, v\rangle\eqdef u(v)\inR$ that are continuous on $\Spc X$. It is a Banach space equipped with the dual norm
\begin{align}
\|u\|_{\Spc X'}\eqdef\sup_{\revise{v \in \Spc X}\backslash\{0\}}\frac{\langle u,v\rangle}{\|v\|_{\Spc X}}.
\end{align}
A direct implication of this definition is the generic duality bound
\begin{equation}
\label{Eq:DualBound2}
|\langle u, v\rangle|\le  \|u\|_{\Spc X'} \|v\|_{\Spc X},
\end{equation}
for any $u \in \Spc X, \revise{v \in \Spc X}'$.
In fact, \eqref{Eq:DualBound2} can be interpreted as the Banach generalization of the Cauchy-Schwarz inequality for Hilbert spaces. By invoking the Hahn-Banach theorem, one can also prove that the duality bound is sharp
for any dual pair $(\Spc X,\Spc X')$ of Banach spaces  \cite{Rudin1991}.
This remarkable property inspired
Beurling and Livingston to introduce the notion of duality mapping and to identify conditions of uniqueness \cite{Beurling1962}.
We like to view the latter as the generalization of the classical Riesz map $\Op R: \Spc H' \to \Spc H$ or, rather, its inverse $\Op J_\Spc H=\Op R^{-1}: \Spc H \to \Spc H'$, which describes the isometric isomorphism between a Hilbert space $\Spc H$ and its continuous dual $\Spc H'$ \cite{Riesz1927}. The caveat with Banach spaces is that the duality mapping is
not necessarily bijective nor even single-valued.
%Specifically, the Hahn-Banach theorem ensures that, for any $f\in \Spc X$, there exists some $f^\ast \in \Spc X'$ (the Banach conjugate of $f$) such that $\|f\|_{\Spc X}=\|f^\ast\|_{\Spc X'}$ and the inequality \eqref{Eq:DualBound2} is saturated.

\begin{definition}[Duality mapping]
Let $(\Spc X,\Spc X')$ be a dual pair of Banach spaces.
Then, the elements $\revise{v^\ast}\in \Spc X'$ and $\revise{v \in \Spc X}$ form a conjugate pair
if they satisfy: \begin{enumerate}
\item Norm preservation: $\|\revise{v^\ast}\|_{\Spc X'}=\|v\|_{\Spc X}$, and 
\item Sharp duality bound: $\langle \revise{v^\ast},v \rangle_{\Spc X' \times \Spc X}=\|\revise{v^\ast}\|_{\Spc X'}\|v\|_{\Spc X}$
\end{enumerate}
For any given $v\in \Spc X$, the set of admissible conjugates defines the duality mapping 
$$
\mathcal{J}_\Spc X(v)=\{\revise{v^\ast} \in \Spc X': \|\revise{v^\ast}\|_{\Spc X'}=\|v\|_{\Spc X}  \mbox{ and }  \langle \revise{v^\ast},v \rangle_{\Spc X' \times \Spc X}=\|\revise{v^\ast}\|_{\Spc X'}\|f\|_{\Spc X}\},
%\arg \sup_{\|g\|_{\Spc X'}=\|g\|_{\Spc X'}: g\in \Spc X'}\langle g, f\rangle_{\Spc X' \times \Spc X},
$$
which is a nonempty subset of $\Spc X'$.
Whenever the duality mapping is a singleton (for instance, when $\Spc X'$ is strictly convex), one also
defines the corresponding duality operator
$\Op J_\Spc X: \Spc X \to \Spc X'$, \revise{which is such that
$\mathcal{J}_\Spc X(v)=\{\revise{v^\ast}=\Op J_\Spc X\{v\}\}$}.
%Otherwise, the pairing described by \eqref{Eq:Conjugate} is still valid for any $f\in \Spc X$ with the twist that $\tilde f$ (resp., $f^\ast$) is no longer unique.
 \end{definition}
%When $\Spc X'$ is strictly convex, the duality map $f\mapsto f^\ast$
%is single-valued \cite[Theorem 1.3, p. 43]{Cioranescu2012} so that we can write
%$f^\ast=\Op \mathcal{J}_\Spc X(f)$ where $\Op J$ is a nonlinear operator %that is continuous 
%$\Spc X \to \Spc X'$. % (since it is norm-preserving).

We now list the properties of the duality mapping that are relevant for our purpose (see \cite{Beurling1962}, \cite[
%Theorem 3.5 p.\ 22, 
Proposition 4.7 p.\ 27, Proposition 1.4, p.\ 43]{Cioranescu2012}, \cite[Theorem 2.53, p.\ 43]{Schuster2012}).
\begin{theorem}[Properties of duality mappings]
\label{Theo:DualityMapping}
Let $(\Spc X,\Spc X')$ be a dual pair of Banach spaces. Then, the following holds:
\begin{enumerate}
\item \label{Item:Existence} Every $\revise{v \in \Spc X}$ admits at least one conjugate $v^\ast\in\Spc X'$.
\item $\mathcal{J}_\Spc X(\lambda v)=\lambda \mathcal{J}_\Spc X(v)$ for any $\lambda \inR$  (homogeneity). 
\item For every $\revise{\revise{v \in \Spc X}}$, the set $\mathcal{J}_\Spc X(v)$ is convex and  weak$^\ast$-closed in $\Spc X'$.
\item The duality mapping is single-valued if $\Spc X'$ is strictly convex; the latter condition is also necessary if $\Spc X$ is reflexive.
\item When $\Spc X$ is reflexive, the duality map is bijective
if and only if both $\Spc X$ and $\Spc X'$ are strictly convex.
\end{enumerate}
\end{theorem}
%Property \ref{Item:Existence} is a direct consequence of the Hahn-Banach Theorem.

The most favorable scenario is covered by Item 5.
In that case, the duality map is invertible with $v=(\revise{v^\ast})^\ast=\Op J_{\Spc X'}\Op J_\Spc X\{v\}$; that is, $\Op J_\Spc X^{-1}=\Op J_{\Spc X'}$, in conformity with the property
that $\Spc X''=\Spc X$.

We now prove that the duality map is linear if and only if $\Spc X=\Spc H$ is a Hilbert space. In that case, the unitary operator $\Op J_\Spc H: \Spc H \to \Spc H'$ is precisely the inverse of the Riesz map
$\Op R: \Spc H'\to \Spc H$. %$\Op J: \Spc X \toC \Spc X'$ 
%that describes the isometric isomorphism between a Hilbert space and its dual.

\begin{proposition}
\label{Prop:Hilbert}
Let $(\Spc X,\Spc X')$ be a dual pair of Banach spaces such that $\Spc X'$ is strictly convex. Then, the duality map
$\Op J_\Spc X: \Spc X \to \Spc X', v \mapsto \Op J_\Spc X\{v\}=\revise{v^\ast}$ is linear %(and self-adjoint) 
if and only if $\Spc X$ is a Hilbert space.
%in which case the operator $\Op J$ coincides with the classical Riesz map.
\end{proposition}
\begin{proof} First, we recall that all Hilbert spaces are strictly convex.
Consequently, the indirect part of the statement is Riesz' celebrated representation theorem, which identifies the canonical linear isometry $\Op J_\Spc X=\Op R^{-1}$ between a Hilbert space and its dual \cite{Rudin1991}.
As for the converse implication, we show that the underlying inner product is
\begin{align}
\langle u, v\rangle_{\Spc X}
%&\eqdef\frac{1}{2} \left(\langle \Op \mathcal{J}_\Spc X(x+y),x+y \rangle_{\Spc X'\times \Spc X}-\langle \Op \mathcal{J}_\Spc X(x-y),x-y \rangle_{\Spc X'\times \Spc X}\right)\\
&=\tfrac{1}{2} \langle \Op J_\Spc X\{u\},v \rangle_{\Spc X'\times \Spc X}+\tfrac{1}{2}\langle \Op J_\Spc X\{v\},u\rangle_{\Spc X'\times \Spc X}.
\end{align}
Its bilinearity follows from the bilinearity of the duality product and the linearity of $\Op J_\Spc X$, while the symmetry in $u$ and $v$ is obvious. 
%Moreover, the form is symmetric because
%$\langle y, x\rangle_{\Spc X}=\langle \Op \mathcal{J}_\Spc X(y),x \rangle_{\Spc X'\times \Spc X}=\langle \Op J^\ast(x),y \rangle_{\Spc X'\times \Spc X}=\langle \Op \mathcal{J}_\Spc X(x),y \rangle_{\Spc X'\times \Spc X}=\langle y, x\rangle_{\Spc X}$.
Finally, the definition of the conjugate yields 
\begin{align}\langle v, v\rangle_{\Spc X}=\langle \Op J_\Spc X\{v\},v \rangle_{\Spc X'\times \Spc X}=\langle \revise{v^\ast}, v \rangle_{\Spc X'\times \Spc X}=\|v\|^2_{\Spc X},
\end{align}
which confirms that the bilinear form $\langle \cdot, \cdot\rangle_{\Spc X}$ is positive-semi-definite. Hence, it is the inner product that induces the
$\|\cdot\|_{\Spc X}$-norm.
\end{proof}

As an example, we provide the expression of the (unique) Banach conjugate 
$\revise{v^\ast}=\Op J_\Spc X\{v\} \in L_q(\R^d)$ of a function 
$v \in L_p(\R^d)\backslash \{0\}$ with $1<p<\infty$ and $\frac{1}{p}+\frac{1}{q}=1$:
\begin{align}
\revise{v^\ast}(\V x)=\frac{\left| v(\V x)\right|^{p-1}}{\|v\|^{p-2}_{L_p}} {\rm sign}\big(f(\V x)\big). 
\end{align}
This formula is intimately connected to H\"older's inequality. In particular, 
the $L_2$ conjugation map with $p=q=2$ is an identity.

\subsection{General representer theorem}
\label{Sec:GenTheo}
We now make use of the powerful tool of conjugation to characterize the solution
of a broad class of unconstrained optimization problems in Banach space.
  \revise{Let us note that the result also covers the equality constraint of Problem \eqref{Eq:standardInterp} if one 
selects the barrier functional 
\begin{align*}
E_{\rm equal}(\V y,\V z)=
\begin{cases}
0, & \V y=\V z\\
+\infty, & \mbox{ otherwise. }
\end{cases}
\end{align*}
}

%\begin{lemma}[{\cite[Proposition 8]{Gupta2018}}] Let $\Spc X$ be a Banach space. Then, a functional $F: \Spc X' \to \R^{+}\cup \{+\infty\}$ that is proper, convex, coercive, and weak$^\ast$-lower-semi-continuous is lower bounded and reaches its infimum. Moreover, the set $S= \arg \min_{f \in \Spc X'} F(f)$ is convex and weak$^\ast$-compact.
%\end{lemma}
%This result is not overly surprising since the property of being weak$^\ast$-lower-semi-continuous
%is actually stricter than the requirement of 
%weak-lower-semi-continuity (w.l.s.c.) that is usually encountered in the theory of convex optimization \cite{Ekeland1999}.

\begin{theorem}[General Banach representer theorem]
\label{Theo:GeneralRepBanach}
Let us consider the following setting:
\begin{itemize}
\item A dual pair $(\Spc X,\Spc X')$ of Banach spaces.
\item The analysis subspace $\Spc N_\V \nu={\rm span}\{\nu_m\}_{m=1}^M \subset \Spc X$
with the $\nu_m$ being linearly independent.
\item The linear measurement operator 
$\V \nu: \Spc X' \to \R^M: f \mapsto \big(\langle \nu_1,f \rangle, \dots,\langle \nu_M, f \rangle\big)$ (it is weak$^\ast$ continuous on $\Spc X'$ because 
$\nu_1,\dots,\nu_M\in \Spc X$).
\item The loss functional $E: \R^M \times \R^M \to \R^{+}\cup \{+\infty\}$ \revise{that is proper, weak lower-semi-continuous and convex in its second argument}.
\item Some arbitrary strictly increasing and convex function $\psi: \R^+ \to \R^+$. 
\end{itemize}
Then, for any fixed $\V y\inR^M$, the solution set of the generic optimization problem
\begin{align}
\label{Eq:GenericOptimizationProb}
S=\arg \min_{f \in \Spc X'}  E\big(\V y, \V \nu(f)\big)+ \psi\left( \|f\|_{\Spc X'}\right)
\end{align}
is nonempty, convex, and weak$^\ast$-compact. If $E$ is strictly convex (or if it imposes the 
equality $\V y=\V \nu(f)$), then
any solution $f_0 \in
S\subset \Spc X'$ is a $(\Spc X',\Spc X)$-conjugate of a common
\begin{align}
\nu_0=\sum_{m=1}^M a_m \nu_m\in \Spc N_\V \nu\subset \Spc X
\end{align}
with a suitable set of weights $\V a \inR^M$; i.e., $S\subseteq \mathcal{J}_\Spc X(\nu_0)$.
Moreover, if $\Spc X$ is %reflexive and 
strictly convex and $f\mapsto\psi(\|f\|_{\Spc X'})$ is strictly convex, 
then the solution is unique with $f_0=\Op J_\Spc X\{\nu_0\} \in \Spc X'$ (Banach conjugate of $\nu_0$). % and $\nu_0=f_0^\ast=(\nu_0^\ast)^\ast \in \Spc X$.
 In particular, if $\Spc X$ is a Hilbert space, then
$f_0=\sum_{m=1}^M a_m \nu_m^\ast$, where $\nu_m^\ast$ is the Riesz conjugate of $\nu_m$.
\end{theorem}

The condition of unicity requires the strict convexity of both $\psi: \R^+ \to \R$ and
$f \mapsto \|f\|_{\Spc X'}$. This applies to
 Banach spaces such as $\Spc X'=\big(L_q(\R^d)\big)'=L_p(\R^d)$  (up to some isometric isomorphism) with $1<p<\infty$ and
the canonical choice of regularization $R(f)=\lambda \|f\|_{L_p}^p$ with $\psi(t)=\lambda |t|^p$ being strictly convex. 
While the solution of \eqref{Eq:GenericOptimizationProb} also exists for
Banach spaces
such as $\Spc M(\R^d)=\big(C_0(\R^d)\big)'$ or $L_\infty(\R^d)=\big(L_1(\R^d)\big)'$, the uniqueness is usually lost in such non-reflexive scenarios (see Section \ref{Sec:Sparse}).

\begin{proof}
The proof uses standard arguments in convex analysis together with a dual reformulation of the problem inspired from the interpretation of best interpolation given by Carl de Boor in \cite{DeBoor1976}.\\[-2ex]
\item {\em (i) Existence and Reformulation as a Generalized Interpolation Problem}.\\
First, we recall that the basic properties of (weak lower semi-) 
continuity %, (strict) convexity, 
and coercivity\footnote{ 
The functional
$F: \Spc X \to\R$, where $\Spc X$ is a Banach space, is said to be coercive if $F(f)\to \infty$ as $\|f\|_\Spc X \to \infty$.} 
are
preserved though functional composition. 
% Also, weak continuity is equivalent to continuity in finite dimensions.
The functional $f \mapsto  \|f\|_{\Spc X'}$ is convex, (norm-)continuous and \revise{coercive}
on $\Spc X'$ from the definition of a norm. 
Since $\psi: \R^+\to\R^+$ is strictly increasing and convex, % and its domain is finite-dimensional,
it is necessarily continuous and coercive. This ensures that $f \mapsto \psi\left( \|f\|_{\Spc X'}\right)$ is endowed with the same three basic properties.
The linear measurement operator $\V \nu: \Spc X' \to \R^N$ is continuous on $\Spc X'$
by assumption (\ie  $\nu_m \in \Spc X \Rightarrow \nu_m \in \Spc X''$ because of the canonical embedding of a Banach space in its bidual)
and trivially convex. Since $\V z \mapsto E\big(\V y, \V z\big)$ is lower semicontinuous on $\R^d$ and convex, this implies by composition the lower-semicontinuity and convexity of
$f \mapsto E\big(\V y, \V \nu(f)\big)$.
Consequently, the functional
$
f\mapsto F(f)=E\big(\V y, \V \nu(f)\big)+ \psi\left( \|f\|_{\Spc X'}\right)
$
is (weakly)
lower-semicontinuous, convex, and coercive on $\Spc X'$, which guarantees the existence of the solution (as well as the convexity and closedness of the solution set) by a standard argument in convex analysis \cite%[Theorem XXX]
{Ekeland1999}---see \cite[Proposition 8]{Gupta2018} for the non-reflexive case.
Moreover, unicity is ensured when $f \mapsto F(f)$ is strictly convex
which happens to be the case when both $\V z \mapsto E\big(\V y, \V z\big)$
and $f \mapsto \psi\left( \|f\|_{\Spc X'}\right)$ are strictly convex.

For the general (not necessarily unique) scenario, we take advantage of the strict convexity of $E(\V y, \cdot)$ %\cup\{+\infty\}$ 
to show that all minimizers of $F(f)$ share
a common measurement vector $\V z_0= \V \nu(f_0) \in \R^M$. 
%(The argument is that the existence of distinct $\V z$ would contradict the assumption of strict convexity.)
\revise{To that end, we pick any two distinct solutions $f_i \in S, i=1,2$ with corresponding measurements $\V z_i= \V \nu(f_i)$
and regularization cost $r_i=\lambda \psi(\|f_i\|_{\Spc X'})$. The convexity of $S$ implies that, for any $\alpha \in (0,1)$,
$f=\alpha f_1 + (1-\alpha)f_2 \in S$ with $\V z=\V \nu(f)= \alpha \V z_1 + (1-\alpha)\V z_2$ and $F(f)=F(f_i), i=1,2$.
Let us now assume that $\V z_1 \ne \V z_2$. Then, by invoking the strictly convexity of
 $\V z \mapsto E(\V y,\V z)$ and the convexity of $f \mapsto \lambda \psi(\|f\|_{\Spc X'})$, we get that
\begin{align*}
F(f)&=E\big(\V y, \alpha \V z_1 + (1-\alpha) \V z_2 \big) + \lambda \psi\big( \|\alpha f_1 + (1-\alpha) f_2\|_{\Spc X'}\big)\\
&< \underbrace{\alpha E\left(\V y, \V z_1\right) +  (1-\alpha) E\left(\V y, \V z_2\right) + \alpha r_1 + (1-\alpha) r_2}
_{\alpha F(f_1) + (1-\alpha) F(f_2)=F(f)},
\end{align*}
which is a contradiction. 
It follows that $\V z_i= \V \nu(f_i)=\V z_0$ which, in turn, implies that the optimal regularization cost $r_i=r_0$ is the same for all $f_i \in S$.}
Although $\V z_0= \V \nu(f_0) \in \R^M$ is usually not known before hand, this property 
provides us with a convenient parametric characterization of the solution set as
\begin{align}
\label{Eq:GInterpol}
S_{\V z}=\arg \min_{f \in \Spc X'} \|f\|_{\Spc X'} \mbox{ s.t. } \V \nu(f)=\V z,
\end{align}
where $\V z$ ranges over $\R^M$. \revise{In this reformulation, we also exploit the property that the minimization of
$\|f\|_{\Spc X'}$ is equivalent to that of $\psi(\|f\|_{\Spc X'})$ because the mapping between the two quantities is monotone.}
\\[-2ex]
\item 
{\em (ii) Explicit Resolution of the Generalized Interpolation Problem \eqref{Eq:GInterpol}}.\\
%Problem \eqref{Eq:GInterpol} to be well-defined for
%any $\V z\inR^M$, we need the 
\revise{The linear independence of the functionals
$\nu_m$  ensures} that any $\nu\in \Spc N_\V \nu$ %\subset \Spc X$
 has the unique expansion $\nu=\sum_{m=1}^M a_m \nu_m$. 
  Based on this representation, we define the
linear functional 
$$
\nu\mapsto\revise{\zeta}(\nu)=\sum_{m=1}^{M} a_m z_m
$$
with $\V z=\V z_0$ fixed.
By construction,
$\revise{\zeta}$ is continuous $\big(\Spc N_\V \nu,\|\cdot\|_{\Spc X}\big)\toC \R$ with
$|\revise{\zeta}(\nu)|\le \|\revise{\zeta}\|\; \|\nu\|_{\Spc X}$,
where
$
\|\revise{\zeta}\|=\sup_{\nu \in \Spc N_\V \nu:\; \|\nu\|_{\Spc X}= 1} \revise{\zeta}(\nu)<\infty$.
Moreover, the Hahn-Banach theorem 
ensures the existence of a continuous, norm-preserving extension of $\revise{\zeta}$ to the whole Banach space $\Spc X$; that is, an element $f_0\in \Spc X'$ such that
$$
\|f_0\|_{\Spc X'}=\sup_{g \in \Spc X:\; \|g\|_{\Spc X}=1} \langle f_0, g\rangle=\|\revise{\zeta}\|.
$$
The connection between the above statement and the generalized interpolation problem \eqref{Eq:GInterpol}
is that the complete set of continuous extensions of $\revise{\zeta}$ to $\Spc X\supset \Spc N_{\V \nu}$ is given by
$$
U=\{f \in \Spc X': \langle f,\nu\rangle=\revise{\zeta}(\nu) \mbox{ for all } \nu \in \Spc N_{\V \nu}\}
$$
with the property that
\begin{align}
\label{Eq:SolutionSet}
f_0\in\arg \inf_{f\in U} \|f\|_{\Spc X'}=S_{\V z_0} \quad \Leftrightarrow \quad  %\inf_{f\in F} \|f\|_{\Spc X'}=
\|f_0\|_{\Spc X'}=\|\revise{\zeta}\|.
\end{align}
The next fundamental observation is that $\Spc N_{\V \nu}=\big(\Spc N'_{\V \nu}\big)'$ because both spaces are of finite dimension \revise{$M$} and, hence, reflexive.
Consequently, for any $\nu_0 \in \mathcal{J}_\Spc X(\revise{\zeta}) \subseteq \big(\Spc N'_{\V \nu}\big)'=\Spc N_\V \nu$, we have
that $\|\nu_0\|_{\Spc X}=\|\revise{\zeta}\|$ and $\revise{\zeta}(\nu_0)=\|\nu_0\|^2_{\Spc X}$, as well as
$\|\nu_0\|_{\Spc X}=\|f_0\|_{\Spc X'}$ for all $f_0 \in S_{\V z_0}$ because of \eqref{Eq:SolutionSet}.
Since $f_0\in U\subset \Spc X'$ and $\nu_0 \in \Spc N_{\V \nu} \subset \Spc X$, this yields
$$
\langle f_0,\nu_0\rangle=\lambda(\nu_0)
=\|f_0\|_{\Spc X'}\|\nu_0\|_{\Spc X},
$$
which implies that $f_0 \in \mathcal{J}_\Spc X(\nu_0)$ with $\mathcal{J}_\Spc X$ the duality mapping from $\Spc X$ to $\Spc X'$.\\[-2ex]
\item 
{\em (iii) Structure of the Solution Set.}\\
We have just shown that $S_{\V z_0}\subseteq \mathcal{J}_\Spc X(\nu_0)$ for any extremal element $\nu_0 \in \{g \in \Spc N_{\V \nu}:
\revise{\zeta}(g)=\|\revise{\zeta}\|\, \|g\|_{\Spc X}, \|g\|_{\Spc X}=\|\revise{\zeta}\|\}$. 
%(In the case where the duality mapping $\mathcal{J}_\Spc X(\lambda)$ is multi\-valued, there is little hope that
%the inclusion goes the other way around because one generally has that  $\mathcal{J}_\Spc X(\nu_0)\ne \mathcal{J}_\Spc X(\nu'_0)$ for $\nu_0\ne\nu_0'$.
%However, we suspect that $S_{\V z_0}=\left(\bigcap_{\nu_0\in \mathcal{J}_\Spc X(\lambda)} \mathcal{J}_\Spc X(\nu_0)\right)$.)
We now %easily 
deduce that $S_{\V z_0}$ is weak$^\ast$-compact since it is included in
the closed ball in $\Spc X'$ of radius $\|f_0\|_{\Spc X'}<\infty$, which is itself weak$^\ast$-compact, by the Banach-Alaoglu theorem.

%
%
%
%By definition of the duality mapping, any solution $f_0$ of \eqref{Eq:GInterpol} with $\V z=\V z_0$ is such that
%$f_0\in \mathcal{J}_\Spc X(\nu_0)$, which implies that $S_{\V z_0}\subseteq \mathcal{J}_\Spc X(\nu_0)$.
%
%Conversely, the extremum $\nu_0 \in \Spc N_{\V \nu}\subset \Spc X$ of $\lambda$ identified
%in Item (ii) is such that $\nu_0\in \mathcal{J}_\Spc X(f_0)$ for all $f_0\in S_{\V z_0}$.
%In fact, there is a full set of equivalent extrema, which is such that 
%$$\mathcal{J}_\Spc X(\lambda)=\{\nu\in \Spc N_{\V \nu}: \lambda(\nu)=\|\lambda\|^2=\|f_0\|^2_{\Spc X}  \}=\Spc N_\V \nu \cap \left(\bigcap_{f_0\in S_{\V z_0}} \mathcal{J}_\Spc X(f_0)\right).
%$$
%The latter a closed compact subset of $\Spc N_{\V \nu}$ which has a most $M$ extreme points; it reduces to a single point when $\Spc X$ is strictly convex.
%
%By Theorem \ref{Theo:DualityMapping}, (Faux!) $S_{\V z_0}=\mathcal{J}_\Spc X(\nu_0)$ is convex and weak$\ast$-closed.
%Moreover, since it is a closed subset of the Banach ball of radius $\|\lambda\|$, which is weak$\ast$ compact (by the Banach-Alaoglu Theorem), $S_{\V z_0}$ is weak$\ast$ compact as well.
When $\Spc X'$ is strictly convex, the situation is simpler because the duality mapping from $\Spc X$ to $\Spc X'$ is single-valued and the solution $f_0\in \Spc X'$ is unique. 
%Likewise, $\nu_0$---the Banach conjugate of $\lambda$ in $(\Spc N_\V \nu,\|\cdot\|_{\Spc X})$---is unique if the $\|\cdot\|_{\Spc X}$-norm is strictly convex. Whenever $\Spc X$ is reflexive, this ensures that there is bijection between $f_0$ and $\nu_0$ with
%$f_0=\nu_0^\ast$ and $\nu_0=f_0^\ast=(\nu_0^\ast)^\ast$.
Moreover, the latter conjugate map is linear if and only if $\Spc X$ is a Hilbert space, by Proposition \ref{Prop:Hilbert}.
\revise{\qed}
\end{proof}

Note that the existence of the conjugate of $\nu_0 \in \Spc N_\V \nu \subset\Spc X$ is essential to the argumentation.
This is the reason why the problem is formulated with $f \in \Spc X'$ subject to the hypothesis
that $\nu_1,\dots,\nu_M \in \Spc X$ (weak$^\ast$ continuity). These considerations are inconsequential
in the simpler reflexive scenario where the role of the two spaces is interchangeable since $\Spc X=\Spc X''$. 
\revise{The hypothesis of linear independence of the $\nu_m$ in Theorem \ref{Theo:GeneralRepBanach} is only made for convenience. When it does not hold, one can adapt the proof by picking a basis of $\Spc N_{\V \nu}$ of reduced dimension $M'<M$, which then leads to a corresponding reduction in the number $M'$ of degrees of freedom of the solution.}

%it can be dropped as explained in the proof, which then leads to a corresponding reduction in the number $M$ of degrees of freedom of the solution.

%\revise{The hypothesis %in Theorem \ref{Theo:GeneralRepBanach} 
%that the functionals
%$\nu_m$ are linearly independent is required in the proof to ensure that 
%Problem \eqref{Eq:GInterpol} is well-defined for
%any $\V z\inR^M$. If this is not the case, it suffices to identify a basis of $\Spc N_{\V \nu}$ of reduced dimension $M'<M$ and to adapt the result accordingly.}

In the sequel, as we shall apply Theorem \ref{Theo:GeneralRepBanach} to concrete scenarios, we shall implicitly interpret $f \in \Spc X'$ in \eqref{Eq:GenericOptimizationProb} as a function (or, eventually, a vector) rather than a continuous linear functional on $\Spc X$ (the abstract definition of an element of the dual space).
This is acceptable provided that the defining space $\Spc X'$ is isometrically embedded in some classical function spaces such as $L_p(\R^d)$ because of the bijective mapping (isometric isomorphism) that relates the two types of entities; for instance, there is a unique element of $f \in L_p(\R)$ with 
$p$ the conjugate exponent of $q\in[1,\infty)$ such that the linear functional $\zeta \in \big(L_q(\R^d)\big)'$ can be specified as
$\zeta(g)= \langle f,g\rangle=\int_{\R^d} f(\V x)g(\V x) \dint \V x$ and vice versa. This allows us to identify $\zeta=\zeta_f$ as $f \in L_p(\R^d)$, while it also gives a precise meaning to identities such as $L_p(\R^d)=\big(L_q(\R^d)\big)'$.

\section{Strictly-Convex Regularization}
\label{Sec:StrictConvex}
The solution of the optimization problem in Theorem \ref{Theo:GeneralRepBanach} 
is unique whenever the Banach space $\Spc X$ (or $\Spc X'$) is reflexive and strictly convex.
This is the setting that has been studied the most in the literature. We now illustrate the unifying
character of Theorem \ref{Theo:GeneralRepBanach} by using it to retrieve the key results in this area; that is, the classical kernel methods for machine learning in RKHS (Section \ref{Sec:RKHS}),
the resolution of linear inverse problems with Tikhonov regularization (Section \ref{Sec.Tik}), and the link with reproducing kernel Banach spaces (Section \ref{Sec:RKBS}).
In addition, we make use of the conjugate map to present a novel perspective on $\ell_p$ regularization for $p>1$ in Section \ref{Sec:ellp}.

\subsection{Kernel/RKHS Methods in Machine Learning}
\label{Sec:RKHS}

Here, the search space $\Spc X'$ is a reproducing-kernel Hilbert space on $\R^d$ denoted by $\Spc H$ with
$\|f\|^2_{\Spc H}=\langle f, f\rangle_{\Spc H}$, where $\langle \cdot, \cdot\rangle_{\Spc H}$ is the underlying inner product. The predual space is $\Spc X=\Spc H'$ which agrees with $\Spc X'=\Spc H''=\Spc H$ (reflexive scenario).
The RKHS property \cite{Aronszajn1950} is equivalent to the existence of a (unique) positive-definite kernel $r_\Spc H: \R^d \times \R^d \to \R$ (the reproducing kernel of $\Spc H$) such that
\begin{align}
%h(\V x_m,\cdot)=
(i) & \ \  r_\Spc H(\cdot,\V x_m) \in \Spc H %\nonumber
\\
\label{Eq:RKprop}
(ii) &\ \ f(\V x_m)=\langle f, r_\Spc H(\cdot,\V x_m)\rangle_{\Spc H}
\end{align}
 for all $f\in \Spc H$ and any $\V x_m\inR^d$. 
 
 In the context of machine learning,
 the loss function $E$ is usually chosen to be additive with
$E(\V y, \V z)=\sum_{m=1}^M E_m\big(y_m, z_m\big)$ \cite{Scholkopf2002} \cite{Hofmann2008}.
Given a series of data points $\big(\V x_m,y_m\big)$, $m=1,\dots,M$ with $\V x_m\inR^d$, the learning problem is then to estimate
a function $f_0: \R^d \to \R$ such that
\begin{align}
\label{Eq:LearningRKHS}
f_0=\arg \min_{f \in \Spc H}  \left(\sum_{m=1}^M E_m\big(y_m, f(\V x_m)\big)  + \lambda \|f\|^2_{\Spc H}\right)
\end{align}
where $\lambda\inR^+$ is an adjustable regularization parameter.
In functional terms, the reproducing kernel represents the Schwartz kernel \cite{Gelfand-Villenkin1964} \cite{Schwartz:1966} of the
Riesz map $\Op R: \Spc H' \to \Spc H: \nu \mapsto \nu^\ast=\int_{\R^d} r_\Spc H(\cdot,\V y) \nu(\V y) \dint \V y$
so that $\nu_m^{\ast}(\V x)=\Op R\{\delta(\cdot-\V x_m)\}(\V x)=r_\Spc H(\V x,\V x_m)$.
The application of Theorem \ref{Theo:GeneralRepBanach} with $\Spc X'=\Spc H$ then immediately yields the parametric form of the solution
\begin{align}
\label{Eq:KernelExp}
f_0(\V x)=\sum_{m=1}^M a_m r_\Spc H(\V x,\V x_m),
\end{align}
which is a linear kernel expansion.
The optimality of such kernel expansions is precisely the result stated in Sch\"olkopf's representer theorem for RKHS \cite{Scholkopf2001}. Moreover, by invoking the reproducing-kernel property
\eqref{Eq:RKprop} with $f=r_\Spc H(\cdot,\V x_n) \in \Spc H$,
one readily finds that
$
\|f_0\|^2_{\Spc H}=\V a^T \M G \V a$,
where the Gram matrix $\M G \in \R^{M \times M}$ is specified by $[\M G]_{m,n}=r_\Spc H(\V x_m,\V x_n)$.
By injecting the parametric form of the solution into the cost functional in
 \eqref{Eq:LearningRKHS}, we then end up with the
equivalent finite-dimensional minimization task
\begin{align}
\label{Eq:KernelDiscrete}
\V a_0=\arg \min_{\V a \in \R^M}  \left( E\big(\V y, \M G\V a)  + \lambda \V a^T \M G \V a\right),
\end{align}
which yields the exact solution of the original infinite-dimensional optimization problem.
%In short, this is the ``miraculous'' aspect of kernel methods: the fact that
%they yield a machinery that can determine the solution of an infinite-dimensional functional optimization problemt, 
In short, \eqref{Eq:KernelDiscrete} is the optimal dicretization of the 
functional optimization problem \eqref{Eq:KernelExp}, which is then readily transcribable into a numerical implementation using standard (finite-dimensional) techniques. 
\subsection{Tikhonov Regularization}
\label{Sec.Tik}
Tikhonov regularization is a classical appro\-ach for dealing with ill-posed linear inverse problems \cite{Tikhonov1963} \cite{Karayiannis1990}.
The goal there is to recover a function $f: \R^d \to \R$
from a noisy or imprecise series of linear measurements
$y_m=\langle \nu_m,f \rangle + \epsilon_m$, where $\epsilon_m$ is the disturbance term.
By using the same functional framework as in Section \ref{Sec:RKHS} with $\nu_1,\dots,\nu_M \in \Spc H'=\Spc X$, and $\Spc X'=\Spc H''=\Spc H$, one formulates the recovery problem as
\begin{align}
\label{Eq:TikProb}
f_0=\arg \min_{f \in \Spc H}  \left(\sum_{m=1}^M |y_m-\langle \nu_m,f \rangle |^2  + \lambda \|f\|^2_{\Spc H}\right).
\end{align}
The application of Theorem \ref{Theo:GeneralRepBanach} then yields a solution that takes the parametric form
\begin{align}
\label{Eq:Tik}
f_0=\sum_{m=1}^M a_m \varphi_m
\end{align}
with $\varphi_m=\Op R\{\nu_m\}$, where $\Op R$ is the Riesz map 
$\Spc H'=\Spc X\to \Spc H=\Spc X'$. 
The next fundamental observation is that the bilinear form
$(\nu_m,\nu_n) \mapsto \langle \nu_m,\Op R\{\nu_n\} \rangle$ is actually the inner product for the dual space $\Spc H'$ leading to $\langle \nu_m,\varphi_n \rangle=\langle \nu_m,\nu_n\rangle_{\Spc H'}$. In fact, by using the property that $\nu_m$ and $\varphi_m=\nu_m^\ast$ are Hilbert conjugates, we have that
\begin{align}
\langle \nu_m,\varphi_n\rangle=\langle \nu_m,\nu_n\rangle_{\Spc H'}=
\langle \nu^\ast_m,\nu^\ast_n\rangle_{\Spc H}=\langle \varphi_m,\varphi_n \rangle_{\Spc H}
\end{align}
which, somewhat remarkably, shows that the underlying system matrix is equal to the Gram matrix of the basis $\{\varphi_m\}$.

Therefore, by injecting
\eqref{Eq:Tik} into the cost functional in \eqref{Eq:TikProb}, we are able to reformulate the initial optimization problem as
the finite-dimensional minimization
\begin{align}
\label{Eq:TikFiniteDim}
\V a_0=\arg \min_{\V a \in \R^M}  \left(\|\V y- \M H\V a\|^2  + \lambda \V a^T \M H \V a\right),
\end{align}
where the system/Gram matrix $\M H\in \R^{M \times M}$ with 
$[\M H]_{m,n}=\langle \nu_m,\varphi_n \rangle=\langle \varphi_m,\varphi_n \rangle_{\Spc H}$ is symmetric positive-definite. By differentiating the quadratic form in \eqref{Eq:TikFiniteDim} with respect to $\V a$ and setting the gradient to zero, we readily derive the very pleasing %surprisingly simple 
closed-form solution
\begin{align}
\V a_0=(\M H\M H\ +\lambda  \M H)^{-1} \M H\V y=(\M H +\lambda \M I)^{-1} \V y
\end{align}
under the implicit assumption that $\M H$ is invertible. We note that the latter is equivalent to the linear independence of the $\varphi_m$ (resp., the linear independence of the $\nu_m$ due to the Riesz pairing).
\subsection{Reproducing Kernel Banach Spaces}
\label{Sec:RKBS}
The concept of reproducing kernel Banach space, which is the natural generalization of RKHS, was introduced and investigated by Zhang and Xu in \cite{Zhang2009} \cite{Xu2019}. Similar to the Hilbertian case, one can identify the RKBS property as follows.

\begin{definition}
\label{Def:RKBS}
A strictly convex and reflexive Banach space $\Spc B$ of functions on $\R^d$ is called a reproducing kernel Banach space (RKBS) if
$\delta(\cdot-\V x) \in \Spc B'$ for any $\V x \inR^d$.
Then, the unique representer $r_\Spc B(\cdot,\V x)=\Op J_{\Spc B'}\{\delta(\cdot-\V x)\}
%=\delta^\ast(\cdot-\V x) 
\in \Spc B$ when  indexed by $\V x$ is called the reproducing kernel of the Banach space.
\end{definition}

It is then of interest to consider the Banach variant of \eqref{Eq:LearningRKHS} that involves a slightly more general regularization term:
Given the data points $\big(\V x_m,y_m\big)_{m=1}^M$, $m=1,\dots,M$, we want to find the unique solution of the optimization problem
\begin{align}
\label{Eq:LearningRKBS}
f_0=\arg \min_{f \in \Spc B}  \left(\sum_{m=1}^M E\big(y_m, f(\V x_m)\big)  + \psi(\|f\|_{\Spc B}) \right)
\end{align}
where the loss function $E: \R \times \R \to \R$ is convex in its \revise{second} argument and the regularization strength modulated by the function $\psi: \R \to \R^+$, which is \revise{convex} and strictly increasing. Since the space $\Spc B$ is reflexive by assumption, the optimization problem falls into the framework of Theorem \ref{Theo:GeneralRepBanach} with $\Spc X=B'$ and $\Spc X'=\Spc B''=\Spc B$ 
%(due to the isometric isomorphism between a reflexive Banach space and its bidual)
and $\nu_m=\delta(\cdot-\V x_m) \in \Spc B', m=1,\dots,M$, where the latter inclusion is guaranteed by the RKBS property. % (see Definition \ref{Def:RKBS}).
We thereby obtain %application of Theorem \ref{Theo:GeneralRepBanach} then yields 
the parametric form of the solution
as
\begin{align}
\label{Eq:LearningSol}
f_0=\Op J_{\Spc B'}\left\{ \sum_{m=1}^M a_m \delta(\cdot-\V x_m)\right\}=\Op J_{\Spc B'}\left\{ \sum_{m=1}^M a_m r^\ast_{\Spc B}(\cdot,\V x_m)\right\}\quad 
\end{align}
with appropriate coefficients $(a_m)\inR^M$, where the expression on the right-hand side has been included in order to make
the connection with the Banach reproducing kernel, as in \cite{Zhang2009} \cite{Zhang2012b}.
Due to the homogeneity and invertibility of the duality mapping (see Theorem  \ref{Theo:DualityMapping}), we have that
$\Op J_{\Spc B'}\left\{ a_m r^\ast_{\Spc B}(\cdot,\V x_m)\right\}=a_m r_{\Spc B}(\cdot,\V x_m)$.
This implies that \eqref{Eq:LearningSol} yields a {\em linear expansion} in terms of kernels {\em if and only} if $M=1$ or if the duality map $\Op J_\Spc B: \Spc B' \to \Spc B$ is linear. We note that the latter condition
together with Definition \ref{Def:RKBS} is equivalent to $\Spc B=\Spc H$ being a RKHS (by Proposition \ref{Prop:Hilbert}), which brings us back to the classical setting of Section \ref{Sec:RKHS}. The same argumentation is also extendable to the vector-valued setting
which has been considered by various authors both for RKHS and RKBS settings \cite{Alvarez2012} \cite{Micchelli2005} \cite{Zhang2013}.
We also like to point our that our analysis is compatible with some recent results of Combettes {\em et
al.\ }\cite{Combettes2018}, where the corresponding conditions of optimality are stated using subdifferentials.

\subsection{Towards Compressed Sensing: $\ell_p$-Norm Regularization}
\label{Sec:ellp}
A classical problem in signal processing is to recover an unknown discrete signal $\V s %=(s_n) 
\inR^N$
from a set of corrupted linear measurements $y_m=\V h^T_m\V s + \epsilon_m$, $m=1,\dots,M$.
The measurement vectors $\V h_1,\dots,\V h_M \in \R^N$ specify the
system matrix $\M H={[\V h_1 \ \V h_2 \  \cdots \ \M h_M]^T}$ $ \in \R^{M \times N}$.
When $M$ (the number of measurements) is less than $N$ (the size of the unknown signal $\V s$), the reconstruction problem is {\em a priori} ill-posed, and strongly so when
$M\ll N$ (compressed-sensing scenario).
However, if the original signal is known to be sparse (\ie
$\|\V s\|_{0}\le K_0$ with $K_0<2M$) and the system matrix $\M H$ satisfies some ``incoherence'' properties, then
the theory of compressed sensing provides general guarantees for a stable recovery
 \cite{Foucart2013} \cite{Candes2007} \cite{Donoho2006}.
The computational strategy then is to impose an $\ell_p$ regularization (with $p$ small to favor sparsity) on the solution and to formulate the reconstruction problem as
\begin{align}
\label{Eq:lpmin}
\V s = \arg \min_{\V x \in \R^N}  \left( E\big(\V y, \M H\V x)  + \lambda  \|\V x\|^p_{\ell_p}\right)
\end{align}
with $\|\V x\|_{\ell_p}\eqdef\left(\sum_{n=1}^N |x_n|^p\right)^{1/p}$. The traditional choice for compressed sensing is $p=1$, which is the smallest exponent that still results in a convex optimization problem.

We now show how we can use Theorem \ref{Theo:GeneralRepBanach} to characterize the effect of such a regularization 
for $p\in(1,\infty)$. The corresponding Banach space is $\Spc X'=(\R^N,\|\cdot\|_{\ell_p})$ whose predual is
$\Spc X=(\R^N,\|\cdot\|_{\ell_q})$ with $\frac{1}{p}+\frac{1}{q}=1$. Moreover, the underlying norms are strictly convex for $p>1$, which guarantees that the solution is unique, irrespective of $M$ and $\M H$. By introducing the dual signal
$\V \nu_0=\M H^T\M a \in \Spc X$ and by using the known form of the corresponding Banach $q$-to-$p$ duality map \revise{$\Op J_{\Spc X}: \Spc X \to \Spc X'$}, we then readily deduce that the solution can be represented as
\begin{align}
\label{Eq:lpminrep}
[\V s]_n=\frac{\left| [\M H^T \V a]_n)\right|^{q-1}}{\|\M H^T \V a\|^{q-2}_{\ell_q}} {\rm sign}\big([\M H^T \V a]_n\big)
\end{align}
for a suitable value of the (dual) parameter vector $\V a \inR^M$.
While the exact value of $\V a$ is data-dependent, \eqref{Eq:lpminrep} provides us with the description of the solution manifold of intrinsic dimension $M$.
Another way to put it is that the fact that $\V s$ minimizes \eqref{Eq:lpmin} induces
a nonlinear pairing between the data vector $\M y \inR^M$
and the dual variable $\V a\inR^M$ in \eqref{Eq:lpminrep}.
In particular, for $p=2$, we have that $\V s=\M H^T \V a=\sum_{m=1}^M \V h_m a_m$, which confirms the well-known result that $\V s\in {\rm span}\{\M h_m\}$. The latter also explains why classical quadratic/Tikhonov regularization performs poorly when $M$ is much smaller than $N$.

\section{Sparsity-Promoting Regularization}
\label{Sec:Sparse}
The limit case\footnote{Our analysis is not applicable to $p<1$ because the corresponding
metric no longer fulfills the properties of a norm; in other words, $\ell_p(\Z)$ fails to be a Banach space for $p\in(0,1)$.}
 of the previous scenario is $p=1$ (CS) for which the norm is no longer strictly convex.
To deal with such cases where the solution is potentially non-unique, we first recall the Krein-Milman theorem \cite[p.\ 75]{Rudin1991},
which allows us to describe the weak$^\ast$-compact solution set $S$ in Theorem \ref{Theo:GeneralRepBanach} as the convex hull of its extreme points.
We then invoke a recent result by Boyer et al.\ that yields the following characterization of the
extremal points of Problem \eqref{Eq:GenericOptimizationProb}.

\begin{theorem}
\label{Prop:sumofextremes}
All extremal points \revise{$f_{0,\rm ext}$} of the solution set $S$ of Problem \eqref{Eq:GenericOptimizationProb} can be expressed
as
\begin{align}
\label{eq:sumatoms}
\revise{f_{0,\rm ext}}=\sum_{k=1}^{K_0} a_k e_k
\end{align}
for some $1\le K_0\le M$ where the $e_k$ are some extremal points of the unit ``regularization'' ball 
$B_{\Spc X'}=\{f \in \Spc X': \|f\|_{\Spc X'}\le 1\}$ and $(a_k)\inR^{K_0}$ is a vector of appropriate weights.
\end{theorem}
The above is a direct corollary of \cite[Theorem 1 with $j=0$]{Boyer2018}
applied to an extreme point of the equivalent generalized interpolation problem  \eqref{Eq:GInterpol}. 
We also note that the existence of \revise{a minimizer $f_0\in S$} of the form
\eqref{eq:sumatoms} has been established independently by Bredies and Carioni \cite{Bredies2018} in a framework
that is even more general than the one considered here. The latter property is also directly deducible from the reduced problem \eqref{Eq:GInterpol} and a classical result by Singer {\cite[Lemma 1.3, p.\ 169]{Singer1970}}. It remains that the existence of \revise{a global minimizer} of the form \eqref{eq:sumatoms} is not as strong a result as Theorem \ref{Prop:sumofextremes}, which tells us the characterization applies for all extremal points of $S$.
Moreover, it should be pointed out that the result in Theorem \ref{Prop:sumofextremes} is not particularly informative for strictly convex spaces such as $\ell_p(\Z)$ or $L_p(\R^d)$ with $p\in(1,\infty)$ for which
all unit vectors (\ie $e\in \Spc X'$ with $\|e\|_{\Spc X'}=1$) are extremal points of the unit ball. Indeed, since the corresponding solution is unique (by Theorem \ref{Theo:GeneralRepBanach}), we trivially have that $f_0=\|f_0\|_{\Spc X'}e_1$ with $K_0=1$ and $e_1=f_0/\|f_0\|_{\Spc X'}$.

By contrast, the characterization in Theorem \ref{Prop:sumofextremes}
is highly relevant for the non-strictly convex space
$\Spc X'=\ell_1(\Z)$ whose extreme vectors are intrinsically sparse;
i.e, $e_k=(\pm\delta[n-n_k])_{n\in \Z}$  for some fixed offset $n_k\in \Z$. 
Here, $\delta[\cdot]$ denotes the Kronecker impulse
which is such that $\delta[0]=1$ and 
$\delta[n]=0$ for $n\ne0$.
Hence, the outcome is that the use of the $\ell_1$ penalty (e.g., \eqref{Eq:lpmin} with $p=1$) has a tendency to induce sparse solutions
with $\|f\|_0=K_0\le M$, which is %more or less 
the flavor of the representer theorem(s) in \cite{Unser2016}.
Two other practically-relevant examples that fall in the non-strictly convex category are considered next.

\subsection{Super-resolution Localization of Spikes}
\label{Sec:Spikes}
The space of continuous functions 
over a compact domain $\Omega \subset \R^d$ equipped with the supremum (or $L_\infty$) norm
is a classical Banach space denoted by
\begin{align}
C(\Omega)=\{f : \Omega\to \R: \|f\|_{\infty}\eqdef \sup_{\V x \in \Omega} |f(\V x)|<\infty\}.
\end{align}
Its continuous dual
\begin{align}
\Spc M(\Omega)=\{f: C(\Omega) \to \R: \|f\|_{\Spc M}\eqdef \sup_{\varphi \in C(\Omega):\, \|\varphi\|_\infty\le 1} \langle f, \varphi\rangle < \infty\}
\end{align}
is the Banach space of bounded (signed) Radon measures on $\Omega$ (by the Riesz-Markov representation theorem \cite{Rudin1987}).
Moreover, it is well known that the extreme points of the unit ball in $\Spc M(\Omega)$ are point measures (a.k.a.\ Dirac impulses)
of the form $e_k=\pm\delta(\cdot-\V x_k)$ for some $\V x_k\in\Omega$, with the property that
\begin{align}
 \varphi \mapsto \langle \delta(\cdot-\V x_k), \varphi \rangle= \varphi(\V x_k)
\end{align}
for any $\varphi \in C(\Omega)$.
For a series of (independent) analysis functions
$\nu_1,\dots,$ $\nu_M\in C(\Omega)$ (e.g., Fourier exponentials), we can invoke Theorems \ref{Theo:GeneralRepBanach} and \ref{Prop:sumofextremes}
with $\Spc X'=\Spc M(\Omega)$ 
to deduce that the extreme points of the problem
\begin{align}
\label{Eq:SpikeProb}
S=\arg \min_{f \in \Spc M(\Omega)}  \left( E\big(\V y,\V \nu(f)\big)  + \lambda \|f\|_{\Spc M}\right)
\end{align}
are inherently sparse. This means that there necessarily exists 
at least one minimizer of the form
\begin{align}
\label{Eq:SpikeSol}
f_0=\sum_{k=1}^{K_0} a_k \delta(\cdot-\V x_k)
\end{align}
with $K_0\le M$, $(a_k) \in \R^{K_0}$, and $\V x_1,\dots,\V x_{K_0} \in \Omega$. 
The fact that \eqref{Eq:SpikeProb} admits
a global solution whose representation is given by \eqref{Eq:SpikeSol} is a result that can be traced back to the work of Fisher and Jerome in \cite[Theorem 1]{Fisher1975}. This optimality result is the foundation for a recent variational method for super-resolution localization that was investigated by a number of authors \cite{Bredies2013,Candes2013b,Fernandez2016}. Besides the development of grid-free optimization schemes, researchers have worked out the conditions on $\V x_k$ and $\nu_m$ under which \eqref{Eq:SpikeProb} can provide a perfect recovery of spike trains of the form given by \eqref{Eq:SpikeSol} with a small $K_0$
\cite{Candes2014,Denoyelle2017,Poon2019}. The remarkable finding is that there are many configurations for which  super-resolution recovery is guaranteed, with an accuracy that only depends on the
signal-to-noise ratio and the minimal spacing between neighbouring spikes.

\subsection{Sparse Kernel Expansions}
\label{Sec:SparseKernel}
Schwartz' space of smooth and rapidly decaying functions on $\R^d$ is denoted by
$\Spc S(\R^d)$. Its continuous dual is $\Spc S'(\R^d)$: the space of 
tempered distributions.
In this section, $\Lop: \Spc S'(\R^d) \toC \Spc S'(\R^d)$
is an invertible operator with continuous inverse
$\Lop^{-1}: \Spc S'(\R^d) \toC \Spc S'(\R^d)$. We also assume
that the generalized impulse response of $\Lop^{-1}$ is a bivariate function of slow growth
$h: \R^d \times \R^d \to \R$. In other words,
the inverse operator $\Lop^{-1}$ has the explicit integral representation 
\begin{align}
\label{Eq:invkernel}
\Lop^{-1}\{\varphi\}=\int_{\R^d} h(\cdot,\V y) \varphi(\V y) \dint \V y
\end{align}
for any $\varphi \in \Spc S(\R^d)$.
In conformity with the nomenclature of \cite{Unser2017}, 
the native Banach space for $\big(\Lop,\Spc M(\R^d)\big)$ is
\begin{align}
\Spc M_\Lop(\R^d)=\{f\in \Spc S'(\R^d): \|\Lop f\|_{\Spc M}\eqdef \sup_{\varphi \in \Spc S(\R^d): \, \|\varphi\|_\infty\le 1} \langle \Lop f, \varphi\rangle < \infty\}.
\end{align}
It is isometrically isomorphic to
$\Spc M(\R^d)$ (the space of bounded Radon measures on $\R^d$). This is to say that the operators $\Lop, \Lop^{-1}$ have restrictions $\Lop: \Spc M_\Lop(\R^d) \toC \Spc M(\R^d)$ and 
$\Lop^{-1}: \Spc M(\R^d) \toC \Spc M_\Lop(\R^d)$ that are isometries. 
Consequently, we can apply Theorem \ref{Theo:GeneralRepBanach} to deduce that the generic learning problem
\begin{align}
\label{Eq:sparsekernelProb}
S=\arg \min_{f \in \Spc M_\Lop(\R^d)}  \left(\sum_{m=1}^M E_m\big(y_m, f(\V x_m)\big)  + \lambda \|\Lop f\|_{\Spc M}\right)
\end{align}
admits a solution, albeit not necessarily a unique one since the underlying search space
$\Spc M_\Lop(\R^d)$---or, equivalently, the parent space $\Spc M(\R^d)$---is neither reflexive nor strictly convex.

In order to refine the above statement
with the help of Theorem  \ref{Prop:sumofextremes}, we 
first observe that the extreme points of the unit ball in $\Spc M(\R^d)$
take the form $e_k=\pm \delta(\cdot-\V \tau_k)$ with $\V \tau_k\inR^d$, which is consistent with the result in Section \ref{Sec:Spikes} for $\Spc M(\Omega)$.
Since the map $\Lop^{-1}: \Spc M(\R^d) \toC \Spc M_\Lop(\R^d)$
is isometric, this allows us to identify the extreme points of the unit ball in $\Spc M_\Lop(\R^d)$ as
\begin{align}
u_k=\Lop^{-1}\{e_k\}=\pm \Lop^{-1}\{\delta(\cdot-\V \tau_k)\}=\pm h(\cdot,\V \tau_k)
\end{align}
where $h: \R^d \times \R^d \to \R$ is the kernel of the operator in
\eqref{Eq:invkernel}. Consequently, we can invoke
Theorem \ref{Prop:sumofextremes} 
to prove that the extreme points of Problem \eqref{Eq:sparsekernelProb}
are all expressible as
\begin{align}
\label{Eq:sparsekernel}
f_0(\V x)=\sum_{k=1}^{K_0} a_k h(\V x,\V \tau_k)
\end{align}
with parameters $K_0\le M$, $\V \tau_1,\dots,\V \tau_{K_0} \inR^d$, and $(a_k)\inR^{K_0}$.
Moreover, since $\Lop\{h(\cdot,\V \tau_k)\}=\delta(\cdot-\V \tau_k)$ and
$\|\delta(\cdot-\V \tau_k)\|_{\Spc M}=\|e_k\|_{\Spc M}=1$, the optimal regularization cost is
$\|\Lop f_0\|_{\Spc M}=\sum_{k=1}^{K_0}|a_k|=\|\V a\|_{\ell_1}$, which makes an interesting connection
with $\ell_1$-norm minimization and the generalized LASSO \cite{Tibshirani1996} \cite{Roth2004}.
To sum up, the solution \eqref{Eq:sparsekernel} has a kernel expansion that is similar to \eqref{Eq:KernelExp}, with the important twist that the kernel centers $\V \tau_k$ are adaptive, meaning that their location as well as their cardinality $K_0$ is data-dependent. In effect, it is the underlying $\ell_1$-norm penalty that helps reducing the number $K_0$ of active kernels, thereby producing a sparse solution. 
\revise{We should also point out that the form of the solution is compatible with the empirical method of moving and learning the data centers in kernel expansions (see \cite[Section IV]{Poggio1990b}) with the important difference that the present proposal is purely variational}.

When $\Lop: \Spc S'(\R^d) \toC \Spc S'(\R^d)$ is linear shift-invariant (LSI) with frequency response $\Fourier\big\{\Lop\delta\big\}(\bw)=\widehat L(\bw)$, then
$h(\V x,\V \tau)=h_{\rm LSI}(\V x-\V \tau)$ with 
\begin{align}
\label{Eq:kernelFourier}
h_{\rm LSI}(\V x)=\Fourier^{-1}\left\{\frac{1}{\widehat L(\bw)}\right\}(\V x),
\end{align}
where the operator $\Fourier^{-1}: \Spc S'(\R^d) \to \Spc S'(\R^d)$
is the generalized inverse Fourier transform.

The overarching message in the optimality result of the present section is that the choice of the regularization operator $\Lop$ in \eqref{Eq:sparsekernelProb} predetermines the parametric form of the kernel in \eqref{Eq:sparsekernel}. Now, in light of \eqref{Eq:kernelFourier}, we can choose to specify first a kernel 
$h_{\rm LSI}: \R^d \to \R$ and then infer the frequency response of the corresponding regularization operator
\begin{align}
\label{Eq:frequencyresponse}
\widehat L(\bw)=\frac{1}{\widehat h_{\rm LSI}(\bw)}.
\end{align}
Now, the necessary and sufficient condition for the continuity of $\Lop:  \Spc S'(\R^d) \to \Spc S'(\R^d)$
is that the function $\widehat L: \R^d \to \R$ be smooth and slowly growing \cite{Schwartz:1966}.
A parametric class of kernels that meets this admissibility requirement is the super-exponential family
\begin{align}
h_{\rm LSI}(\V x)=\exp\left(-\|\V x\|^{\alpha}\right)
\end{align}
with $\alpha\in(0,2)$. The limit case with $\alpha=2$ (Gaussian) is excluded
because the corresponding frequency response in \eqref{Eq:frequencyresponse} (inverse of a Gaussian) fails to be slowly increasing.

\section{Conclusion}
We have shown that the fundamental ingredient in the quest for a representer theorem is the identification and characterization of a dual pair of Banach spaces that is linked to the regularization functional. 
%The main point of this paper has been to show that the general issue of regularization can be investigated through a common (abstract) umbrella. 
The main result of the paper is expressed by Theorem 
\ref{Theo:GeneralRepBanach}, which is valid for Banach spaces in general. %, which is often simpler, depending on the context.
This characterization of the solution of the general optimization problem \eqref{Eq:GenericOptimizationProb} is directly exploitable in the reflexive and strictly convex scenario---in which case the solution is also known to be unique---whenever the duality mapping is known.
% (e.g., for the classical $L_p(\R^d)$ or $\ell_p(\Z)$ spaces with $p\in(1,\infty)$).
While our formulation also offers interesting insights for certain non-strictly convex and sparsity-promoting norms such as 
$\|\cdot\|_{\ell_1}$ and its continuous-domain counterpart---the total variation $\|\cdot\|_{\Spc M}$ and generalization thereof---it raises intriguing questions about the unicity of such solutions and the necessity to
develop some corresponding numerical optimization schemes.

We have made the link with the existing literature %on representer theorems 
in machine learning (regression) and the resolution of ill-posed inverse problems
by considering several concrete cases, including reproducing kernel Hilbert spaces (RKHS) and compressed sensing. 
The conciseness and self-containedness of the proposed derivations is a good indication of the power of the approach.

Since the concept of Banach spaces is remarkably general, one can easily conceive of other variations around the common theme of regularization and representer theorems.
Potential topics for further research include the use of nonstandard norms, the deployment of hybrid regularization schemes, vector-valued functions or feature maps \cite{Alvarez2012}, and the consideration of direct-sum spaces and semi-norms, as in the theory of splines \cite{deBoor1966} \cite{Duchon1977} \cite{Wahba1990} \cite{rabut:2004} \cite{Mosamam2010} \cite{Unser2017}. In short, there is ample room for additional theoretical and practical investigations, in direct analogy with what has been accomplished during the past few decades in the simpler but more restrictive context of RKHS \cite{Argyriou2009,Alvarez2012}.
Interestingly, there also appears to be a link with deep neural/kernel networks, as has been demonstrated recently \cite{Bohn2018,Unser2019}.

\subsection*{Acknowledgments}
The research was partially supported by the Swiss National Science
Foundation under Grant 200020-162343.
% and the European Commission under Grant ERC-2010-AdG 267439-FUN-SP.
The author is thankful to Julien Fageot, Shayan Aziznejad, Pakshal Bohra, Quentin Denoyelle and Philippe Th{\'e}venaz for helpful comments on the manuscript.

%\section*{Bibliography}

%\bibliographystyle{plain}
\bibliographystyle{spmpsci} 
%\bibliographystyle{siamplain}
%
% To get the path, drag file "Unser.bib" on terminal
%
%\bibliography{/Users/unser/GoogleDrive/Bibliography/Bibtex_files/Unser}
%
% To construct restricted sublibrary,  under tools in Jabref, "New sublibrary based on AUX file"
%
%
\bibliography{Unification}
\end{document}